\documentclass[a4paper,12pt]{article}
\usepackage[english]{babel}
\usepackage{amsmath,amssymb,amsthm}
\usepackage{hyperref}
\usepackage[table]{xcolor}
\definecolor{gry}{gray}{0.75}

 \makeatletter
 \def\@seccntformat#1{\csname the#1\endcsname.\ } 
 \makeatother

\date{}

\newtheorem{theorem}{Theorem}
\newtheorem{lemma}{Lemma}
\newtheorem{proposition}{Proposition}
\newtheorem{corollary}{Corollary}
\theoremstyle{definition}
\newtheorem{problem}{Problem}
\newtheorem{remark}{Remark}

\title{An enumeration of 1-perfect ternary codes%
\thanks{This is the author's final version of the manuscript
 published in Discrete Mathematics 346(7), 2023, paper 113437, \url{https://doi.org/10.1016/j.disc.2023.113437}

The work of M. J. Shi is supported by the National Natural Science
Foundation of China (12071001); the work of
D. S. Krotov is supported within the framework of the state contract of the
Sobolev Institute of Mathematics (FWNF-2022-0017).
}}
\author{%
Minjia Shi%
\thanks{(a)
Key Laboratory of Intelligent Computing and Signal Processing, Ministry of Education, School of Mathematical Sciences, Anhui University,
Hefei, Anhui, China.
(b)
State Key Laboratory of Integrated Service Networks, Xidian University, Xi’an, 710071, China.
\href{mailto:smjwcl.good@163.com}{smjwcl.good@163.com}
}
\ and Denis S. Krotov%
\thanks{(c) Sobolev Institute of Mathematics, Novosibirsk 630090, Russia.
\href{mailto:krotov@math.nsc.ru}{krotov@math.nsc.ru}
}
}

\newcommand\vc[1]{\bar{#1}}
\newcommand\FF{\mathbb{F}}
\newcommand\Aut{\mathrm{Aut}}
\newcommand\MAut{\mathrm{MAut}}
\newcommand\PAut{\mathrm{PAut}}
\newcommand\GL[2][3]{\mathrm{GL}_{#2}(\FF_{#1})}

\begin{document}
\maketitle

\begin{abstract}
We study codes with parameters of the ternary Hamming $(n{=}(3^m-1)/2,3^{n-m},3)$ code, i.e., ternary $1$-perfect codes. The rank of the code is defined to be the dimension of its affine span. We characterize ternary $1$-perfect codes of rank $n-m+1$, count their number, and prove that all such codes can be obtained from each other by a sequence of two-coordinate switchings. We enumerate ternary $1$-perfect codes of length $13$ obtained by concatenation from codes of lengths $9$ and $4$; we find that there are $93241327$ equivalence classes of such codes.

\end{abstract}

\section{Introduction}

Perfect $1$-error-correcting $q$-ary codes are codes
with parameters of $q$-ary Hamming codes over the 
Galois field GF$(q)$ of order $q$,
which exist for every prime power $q$
and length $n$ of form $(q^m-1)/(q-1)$,
$m\in\{2,3,\ldots\}$.
Since the pioneer work of Vasil'ev on 
$1$-perfect binary codes~\cite{Vas:nongroup_perfect.de}
and its $q$-ary generalization 
by Sch{\"o}nheim~\cite{Schonheim68},
it is known that the Hamming code is not a unique
$1$-perfect code and the number of nonequivalent
$1$-perfect codes grows doubly exponentially in $n$
(at least $q^{q^{cn-o(n)}}$, 
where $c=\frac1q$ if $q=2,3$ 
 and 
$c\simeq \frac2q$ for large~$q$~\cite{HK:q-ary}).

Perfect codes, including non-binary ones, 
can be used in different applications,
for example, in 
steganography schemes, see e.g.~\cite{ZhangLi:2008}, \cite{RRR:2011}, \cite{KPCh:2019}.
As mentioned in \cite{ZhangLi:2008},
the possibility to choose a code from a large variety 
can increase security of steganographic systems,
adding additional difficulties to anyone who
wants to hack such a scheme. 
So, the study of nonlinear $1$-perfect codes,
especially classes of codes whose structure is well understood
(which allows to develop efficient decoding algorithms),
is important both from theoretical point of view and
for evaluating their potential use in applications.

The problem of characterization of the class of $1$-perfect 
codes, in any constructive terms, is far from being solved 
even for $q=2$. However, there are characterizations 
for $1$-perfect codes with some restrictions.
One of important restrictions is the restriction 
on the rank of the code. 
The rank is the dimension of the affine span of the code.
We say that a $1$-perfect code is of rank $+r$ if its rank is
$r$ greater than the dimension
of the Hamming code of the same parameters.
Binary $1$-perfect code of rank at most~$+1$ 
and of rank at most~$+2$
are characterized by Avgustinovich,
Heden, and Solov'eva in~\cite{AvgHedSol:class};
in the last case, the characterization 
is up to the characterization
of multiary quasigroups of order~$4$, 
which was completed 
later in~\cite{KroPot:4}.
In~\cite{HK:q-ary},
all $1$-perfect codes of rank less than
$n-1$ are proven to be decomposed into
some independent subsets, so-called $\mu$-components,
but the characterization of $\mu$-components
is as hard as of $1$-perfect codes in general.
However, in special cases 
(for example, for the case $q=2$ and rank~$+2$ 
considered in~\cite{AvgHedSol:class}),
the structure of $\mu$-components can be well understood,
and we use this decomposition
to characterize 
the set of ternary $1$-perfect codes of rank~$+1$,
which is the first main result of the current paper.

The second result is computational: we classify ternary concatenated $1$-perfect codes of length~$13$. 
It can be considered 
as a ternary analog of the result~\cite{Phelps:2000}
on the classification of binary concatenated 
extended $1$-perfect codes of length~$16$; 
however, in contrast to the results of~\cite{Phelps:2000},
the number $93241327$ of nonequivalent ternary perfect codes 
found is huge, and processing them
in a reasonable time (it took about thirty 
core-years)
required combining and development 
of classification methods used in~\cite{Phelps:2000}
for codes and by the second author 
for equitable partitions and orthogonal arrays~\cite{Kro:OA13}.
By now, all known computational enumerations of 
$1$-perfect codes were focused on the binary case.
In contrast to the $q$-ary $1$-perfect codes with $q>2$,
every binary $1$-perfect code has its extended
version obtained by appending 
the all-parity-check bit to every codeword.
Although this mapping is bijective,
enumerating $1$-perfect and
extended $1$-perfect codes up to equivalence
are different tasks.
Phelps~\cite{Phelps:2000}
enumerated partitions of
the binary Hamming space of dimension~$7$ 
into 
$1$-perfect codes and binary extended
$1$-perfect codes of length~$16$
obtained from such partitions by concatenation.
V.\,Zivoniev and D.\,Zinoviev 
enumerated binary $1$-perfect codes of
length~$15$ and rank~$13$ in~\cite{ZZ:2004}, 
extended binary $1$-perfect codes of
length~$16$ and ranks~$13$ and~$14$ 
in~\cite{ZZ:2002} and~\cite{ZZ:2006:l16r14}, 
respectively.
Finally, {\"O}sterg{\aa}rd and Pottonen
enumerated all binary $1$-perfect codes of length~$15$
and their extended versions in~\cite{OstPot:15};
in the subsequent paper~\cite{OPP:15}, 
different properties
of these codes were studied.
Note that, as was mentioned 
in the later researches \cite{ZZ:2006:l16r14}, \cite{OPP:15},
the numbers of nonequivalent codes
found in~\cite{Phelps:2000}, \cite{ZZ:2002}, \cite{ZZ:2004},
and~\cite{ZZ:2006:l16r14} contain mistakes;
however, methods
developed there are correct
and important for further research,
including our current study.

In contrast to the most of previous
results on non-binary $1$-perfect codes,
e.g.,
\cite{PheVil:2002:q},
\cite{Los06},
\cite{Mal2010:q},
\cite{HK:q-ary},
\cite{Romanov:2019},
in the current paper
we focus our efforts on the ternary case.
It should be noted that this case is of special
interest by the following theoretical reasons.
For~$q=3$,
as well as for~$q=2$, the group of isometries
of the Hamming space is a subgroup of the group
of the automorphisms of the corresponding
affine space over~GF$(q)$. We hence can ensure
that codes that are equivalent combinatorially 
(the equivalence is defined in Section~\ref{ss:equi})
are equivalent algebraically. In particular,
algebraic properties such as the rank, the kernel structure
(see Section~\ref{ss:gs}), 
maximal affine subspaces are invariant under equivalence.
This is not the case for~$q\ge 4$, where a linear code
can be equivalent to a code that linearly spans the whole space.
The last fact is used in~\cite{Dinitz:06} to construct 
special decompositions of a group 
(the additive group of the space) called tilings, 
with different parameters (of special interest are so-called
full-rank tilings, which can be constructed from $1$-perfect codes
spanning the space; the kernel size is also important for tilings). 
However, for~$q \le 3$,
by the reasons mentioned above,
we cannot construct different tilings from equivalent codes.
This motivates to study the connection between perfect codes
and tilings more closely for~$q=2$ and~$q=3$. 
In the binary case, 
most of related questions have been solved
\cite{CLVZ:tiling}, \cite{EV:98}, \cite{TraVar:2003}, \cite{OstVar:2004}, 
including the characterization of all admissible
rank--kernel dimension pairs 
for $1$-perfect codes~\cite{ASH:RankKernel}.
For $q=3$ 
(we mention important results~\cite{OstSza:2007}, 
\cite{Szabo:2008}, 
\cite{Krotov:3tiling}
focused on this case), many questions are open, and the table
of ranks and kernels for concatenated codes (Table~\ref{t:rk})
can be considered as a step in this direction.
\smallskip

The structure of the paper is as follows.
In Section~\ref{s:def}, we define the main concepts
and recall some facts we use in our study.

In~Section~\ref{s:+1}, we prove a characterization
theorem for $3$-ary $1$-perfect codes of rank~$+1$
(Theorem~\ref{th:q3r1}),
count their number (Theorem~\ref{th:all}), 
prove the connectedness 
of the set of such codes by mean of two-coordinate switching 
(Theorem~\ref{th:swi}),
and consider the structure 
of a $3$-ary $1$-perfect code~$C$ 
with kernel size~$|C|/3$ (Section~\ref{s:ker}).

In~Section~\ref{s:class} , we describe 
the computer-aided enumeration of 
the concatenated $3$-ary
$1$-perfect codes of length~$13$
(Theorem~\ref{th:p13})
and of auxiliary objects including the partitions 
of the space into $3$-ary $1$-perfect codes of length~$4$
(Theorem~\ref{th:p4})
and partitions of a $(9,3^8,2)_3$ MDS code into 
$(9,3^6,3)_3$ subcodes (Theorem~\ref{th:RM9}).
Enumerating the last partitions was the most 
resource-intensive step of the computing;
it is of independent interest because 
$(n=q^m,q^{n-m-1},q)_q$ subcodes of an
$(n,q^{n-1},2)_q$ MDS code form an interesting
class of completely regular codes,
which share with perfect codes some properties and 
constructing tools, see e.g.~\cite{Romanov:2022}.
Moreover, the obtained partitions can further be used for 
constructing $1$-perfect codes of any admissible length 
larger than~$13$
by generalized concatenated construction
(see~\cite{Zin1976:GCC} for the general approach)
as shown in~\cite{Romanov:concat}.
A database containing representatives of the equivalence
classes of the classified objects can be found at
\url{https://ieee-dataport.org/open-access/perfect-and-related-codes}%
~\cite{Perfect-related}.


\section{Preliminaries}\label{s:def}
In this section, we define main concepts and mention 
related facts important to our study.
\subsection{Graphs and spaces}\label{ss:gs}
The \emph{Hamming graph} $H(n,q)$ 
is a graph whose vertices 
are the words of length~$n$ in the alphabet 
$\{0,\ldots,q-1\}$, two vertices being adjacent
if they differ in exactly one symbol.
If $q$ is a prime power, the symbols of the alphabet are associated 
with the elements of the prime field
GF$(q)$, and the vertex set of
$H(n,q)$ forms an $n$-dimensional vector space
$\mathbb{F}_q^n$
over GF$(q)$ with the component-wise addition 
and multiplication by a constant.

The natural shortest-path distance in $H(n,q)$
coincides with the \emph{Hamming distance}, i.e.,
the distance between two words equals the number
of positions they differ. 
The \emph{weight} of a vertex~$\vc{x}$ is the distance
from~$\vc{x}$ to the all-zero word~$\vc{0}$.

A vertex set~$C$ in $H(n,q)$
is called a \emph{distance-$d$ code}, or an $(n,|C|,d)_q$
code, if there are no two different codewords in~$C$
with distance less than~$d$.
A code forming a linear subspace of~$\mathbb{F}_q^n$
is called \emph{linear}. 
The \emph{rank} of a code is the dimension of
its affine span (if the code contains the all-zero word, then, equivalently, the dimension of its linear span).
The \emph{kernel} of a code $C\subset \FF_q^n$ is the set 
$\{\vc{x}\in \FF_q^n:\ \alpha\vc{x}+C=C\ \forall \alpha\in\mathbb{F}_q \}$;
if $q$ is prime (in our case, $q=3$),
then
the kernel coincides with the set
$\{\vc{x}\in \FF_q^n:\  \vc{x}+C=C  \}$
of all periods of~$C$.

\subsection{Equivalence and automorphisms}\label{ss:equi}

The next group of definitions concerns different equivalences
and automorphisms of codes. 
We recall that every automorphism
of the graph
$H(n,q)$ can be uniquely represented as the composition 
of a coordinate permutation
$$ \pi: 
(c_0,\ldots,c_{n-1}) \to 
(c_{\pi^{-1}(0)},\ldots,c_{\pi^{-1}(n-1)})
$$
and an \emph{isotopy}
$\vc\theta = (\theta_0,\ldots,\theta_{n-1})$
that consists 
of $n$~permutations of the alphabet
$\{0,\ldots,q-1\}$, acting independently 
on the corresponding $n$ symbols of a word 
of length~$n$ over $\{0,\ldots,q-1\}$:
$$ \vc\theta: 
(c_0,\ldots,c_{n-1}) \to 
(\theta_{0}(c_{0}),\ldots,\theta_{n-1}(c_{n-1})).
$$ 
Two sets~$C$ and~$D$ of vertices of
$H(n,q)$ are said to be \emph{equivalent}
if there is an automorphism of~$H(n,q)$
that sends~$C$ to~$D$. 
The set of automorphism of $H(n,q)$
that send a~vertex set~$C$ to itself forms 
the \emph{automorphism group} $\Aut(C)$ of~$C$,
with composition in the role of the group operation.

Two sets~$C$ and~$D$ of vertices of $H(n,q)$ are
\emph{monomially equivalent}
if there is an automorphism of $H(n,q)$ 
that is 
at the same time an automorphism of the corresponding
vector space
and sends~$C$ to~$D$
(in the case $q=3$, every automorphism of $H(n,3)$
that fixes the all-zero word is 
an automorphism of the vector space).
Two sets~$C$ and~$D$ of vertices of $H(n,q)$ are 
\emph{permutably equivalent} if there is a permutation of 
coordinates that sends~$C$ to~$D$.
The \emph{monomial automorphism group} $\MAut(C)$
and
\emph{permutation automorphism group} $\PAut(C)$
are subgroups of $\Aut(C)$ that correspond to 
monomial and permutation equivalence, respectively.

\subsection{1-perfect, distance-2 MDS, and Reed--Muller-like codes}

A \emph{$1$-perfect code} is an independent 
set of vertices of $H(n,q)$ (or any other graph)
such that every non-code vertex is adjacent to exactly
one codeword.
If $q$ is a prime power, then a necessary and sufficient
condition for the existence of $1$-perfect codes is
$n=(q^m-1)/(q-1)$, $m\in\{1,2,\ldots\}$;
so, such codes are $(n,q^{n-m},3)_q$ codes. 
In particular, for every~$q$ and~$m$,
there is a unique (up to equivalence) linear $1$-perfect code,
called 
a \emph{Hamming code},
which has dimension~$n-m$, 
the order of the 
monomial automorphism group 
$$|\GL[q]{m}|=(q^{m}-1)(q^{m}-q)\ldots(q^{m}-q^{m-1}),$$
and a check matrix consisting of the maximum collection 
of mutually non-colinear columns 
of height~$m$ (recall that the rows of a \emph{check matrix} 
form a basis of the dual space of the linear code).
The Hamming code, obviously, has the minimum rank, $n-m$, among all 
$1$-perfect codes of the same parameters; 
thus,
we say that a $1$-perfect code is \emph{of rank $+r$}
if its rank is $(n-m)+r$.

A code with parameters $(n,q^{n-1},2)_q$ is called a 
\emph{distance-$2$ MDS code} 
(note that we do not require this code to be linear).
A function $f:\{0,\ldots,q-1\}^n \to \{0,\ldots,q-1\}$
such that its graph
$\{ (\vc{x},f(\vc{x})):\ \vc{x}\in \{0,\ldots,q-1\}^n \}$
is a distance-$2$ MDS code 
is called
an \emph{$n$-ary} ({\emph{multiary}) \emph{quasigroup} of order~$q$.

The third special kind of codes that plays a role
in our theory
is $(n=q^m,q^{n-m-1},3)_q$ codes that are subsets
of a distance-$2$ MDS code.
We call such codes
\emph{RM-like codes},
because the linear code of this kind is
$\mathcal{R}_q(qm-m-2,m)$,
a generalized Reed--Muller code
(see, e.g.,~\cite[\S5.4]{AssmKey92})
of order $(q-1)m-2$.
As follows from the following proposition, every RM-like
code is a maximum distance-$3$ subcode of
a distance-$2$ MDS code. 
\begin{proposition}\label{p:RMl}
If $C\subset M$, where $C$ is a RM-like code
and $M$ is a distance-$2$ MDS code,
then every vertex not in~$M$ is adjacent
to exactly one codeword of~$C$.
\end{proposition}
\begin{proof}
Since the minimum distance of~$C$ is~$3$, we see that every
vertex is adjacent to at most one codeword of~$C$.
The number of vertices adjacent to a codeword of~$C$
is $|C|\cdot n \cdot(q-1)$, i.e., $q^n-q^{n-1}$,
which is exactly the number of vertices not in~$M$.
\end{proof}
By an \emph{RM-like partition}, we mean a partition 
of a distance-$2$ MDS code into RM-like codes.

\subsection{Concatenation}
For any two words or symbols $\vc{x}$ and~$\vc{y}$,
by~$\vc{x}\vc{y}$ we denote their concatenation.
For a code~$C$ and a symbol or word~$\vc{x}$, we denote
$C\vc{x}=\{\vc{c}\vc{x}:\ \vc{c}\in C\}$
and
$\vc{x}C=\{\vc{x}\vc{c}:\ \vc{c}\in C\}$;
similarly,
$CD=\{\vc{x}\vc{y}:\ \vc{x}\in C,\ \vc{y}\in D\}$
for two codes~$C$ and~$D$.

Next, we define concatenated codes.
The following construction of $q$-ary $1$-perfect codes
suggested by Romanov~\cite{Romanov:2019}
is a $q$-ary generalization of the Solov'eva--Phelps
construction~\cite{Phelps:83,Sol:81} 
for binary $1$-perfect codes.

\begin{lemma}[Romanov \cite{Romanov:2019}]\label{l:rom}
Assume
$n=(q^m-1)/(q-1)$, $n'=(q^{m-1}-1)/(q-1)$, $n''=q^{m-1}$.
Let $(P_0, \ldots ,P_{n''-1})$ be a partition
of the Hamming space $H(n',q)$ into
$1$-perfect $(n', q^{n'-(m-1)}, 3)_q$ codes.
Let $(C_0, \ldots ,C_{n''-1})$ be a partition
of an $(n'',q^{n''-1},2)_q$ MDS code into $n''$
  codes with parameters
  $(n'', q^{n''-m}, 3)_q$.
And let $\tau$ be a permutation of $\{0, \ldots ,{n''-1}\}$.
Then the code
\begin{equation}
 \label{eq:P}
P= \bigcup_{i=0}^{q^{m-1}-1} C_i P_{\tau(i)} 
\end{equation}
is a $1$-perfect $(n, q^{n-m}, 3)_q$ code.
\end{lemma}

The role of the permutation $\tau$ in the construction above
is technical: since it just changes the order of the codes~$P_i$,
we will not lose generality by assuming that $\tau$ is identity.
However, as in Section~\ref{s:class} 
we work with concrete representatives 
of equivalence classes of partitions,
it is convenient to represent the reordering of the codes
in a partition explicitly, as a permutation~$\tau$.

The codes representable in the form~\eqref{eq:P} are called 
\emph{concatenated}. 
We note that this property is not invariant 
under equivalence because it depends on the 
order of coordinates.

\begin{remark}
Another $q$-ary generalization of the Solov'eva--Phelps
construction was proposed in~\cite{Mollard:84:une}
(see also~\cite[Theorem~11.4.5]{CHLL}); 
it can be regarded as a special case of the construction
in Lemma~\ref{l:rom} with a partition $(C_0,\ldots,C_{n''-1})$
explicitly constructed from a partition 
with the same parameters as $(P_0,\ldots,P_{n''-1})$.
In the case $n''=9$, considered in Section~\ref{s:class},
there are $65436$ nonequivalent partitions $(C_0,\ldots,C_{n''-1})$
of a~$(9,3^8,2)_3$ code into $(9,3^6,3)_3$ codes (see Theorem~\ref{th:RM9}), 
while the number of nonequivalent partitions $(P_0,\ldots,P_{n''-1})$
of~$\FF_3^4$ into $(4,9,3)_3$ codes is only~$2$ (see Theorem~\ref{th:p4}).
This shows that the construction in~\cite{Romanov:2019}
(Lemma~\ref{l:rom} above)
gives more codes than the one in~\cite{Mollard:84:une}.
\end{remark}
\begin{remark}
 The construction in Lemma~\ref{l:rom},
as well as its binary case~\cite{Phelps:83,Sol:81}, 
is very close to the construction 
of Heden~\cite{Heden:77}.
Namely, if we restrict the choice of the partition
$(P_0, \ldots ,P_{n''-1})$ by a partition into cosets
of the same $1$-perfect code and treat the partition
$(C_0, \ldots ,C_{n''-1})$ as a code in the 
mixed-alphabet Hamming space over $\FF_q^{n''} \times\{0,\ldots,n''\}$, 
then the following lemma turns into a special case
of~\cite[Theorem~1]{Heden:77}. 
Finally, we note that the construction can be treated
in terms of the generalized concatenation construction~\cite{Zin1976:GCC}.
\end{remark}

\section{Codes of rank +1}\label{s:+1}
In this section, we characterize 
the $3$-ary $1$-perfect codes of rank~$+1$, 
count the number of different such codes,
and discuss the possibility 
of switching between such codes.
\subsection{Characterization}\label{s:char}
We will use the result of~\cite{HK:q-ary},
which states that a code, depending on its rank,
is the union of one or more independently defined subsets,
called $\vc\mu$-components. Below we will show that in the case
of ternary $1$-perfect codes of rank~$+1$, 
such $\vc\mu$-components are in one-to-one
correspondence with multiary quasigroups of order~$3$.

\begin{lemma}[{\cite[Th.~2.1]{HK:q-ary}, $r=m-1$, $s=1$}] \label{l:HK}
 Let $C$ be a $q$-ary $1$-perfect code of length
 $n=(q^m-1)/(q-1)$ of rank at most~$+1$ and 
 $C^\star$ be the $q$-ary Hamming code
 of length 
 $n'=(q^{m-1}-1)/(q-1)$.
 Then for some translation vector $\vc{v}$ and monomial transformation~$\psi$, it holds
 $$\psi(C+\vc{v}) = \bigcup_{\vc{\mu}\in C^\star} K_{\vc\mu},$$
 where
 $$
   K_{\vc\mu} =
 \big\{(x_0,x_1,\ldots,x_{n-1}): \
 \vc\sigma(x_0, \ldots ,x_{n-2})=\vc\mu,\ \
 x_{n-1} = \lambda_{\vc\mu}(x_0,\ldots,x_{n-2}) \big\},
 $$ 
$$
\vc\sigma(x_0, \ldots ,x_{n-2})
=\bigg(
\sum_{i=0}^{q-1} x_i,  
\sum_{i=q}^{2q-1} x_i, 
\ \ldots ,
\!\!\sum_{i=n-1-q}^{n-2} \!\!\!\!\! x_i 
\bigg),
$$
for some 
$\{0,\ldots,q-1\}$-valued
functions~$\lambda_{\vc\mu}$,
$\vc\mu \in C^*$, defined on 
$$\{(x_0, \ldots, x_{n-2}):\ 
  \vc\sigma(x_0, \ldots ,x_{n-2})=\vc\mu \} $$
and satisfying
\begin{equation}\label{eq:la}
d(\vc{x}_*,\vc{y}_*) = 2 \quad \Longrightarrow \quad
\lambda_{\mu}(\vc{x}_*) \ne \lambda_{\mu}(\vc{y}_*).
\end{equation}
\end{lemma}
In the ternary case, the equation 
 $\vc\sigma(x_0, \ldots ,x_{n-1})=\vc\mu$ can be expressed as follows, 
 where $\vc\mu = (\mu_0, \ldots ,\mu_{n'-1})$:
\begin{equation}\label{eq:xmu} 
 x_2=-x_0-x_1+\mu_0,
 \quad 
 x_5=-x_3-x_4+\mu_1,\ \ldots, 
 \quad 
 x_{n-2}=-x_{n-4}-x_{n-3}+\mu_{n'-1}.
\end{equation}

\begin{lemma}\label{l:q3}
If the hypothesis and the conclusion
of~Lemma~\ref{l:HK} hold with $q=3$,
then 
\begin{equation}\label{eq:la'} 
\lambda_{\vc{\mu}}(x_0, \ldots , x_{n-2})
=
\lambda'_{\vc{\mu}}(x_1-x_0, x_4-x_3, \ldots , x_{n-3}-x_{n-4})
\end{equation}
for some $n'$-ary quasigroup $\lambda'_{\vc{\mu}}$
of order~$3$, where $n'=\frac{n-1}3$. 
Moreover, if $\lambda'_{\vc{\mu}}$ is an arbitrary
$\frac{n-1}3$-ary quasigroup
of order~$3$ and $\lambda_{\vc{\mu}}$ is defined by~\eqref{eq:la'}
on any values of arguments satisfying~\eqref{eq:xmu},
then $\lambda_{\vc{\mu}}$ satisfies~\eqref{eq:la}. 
\end{lemma}
\begin{proof}
 The second claim is straightforward
 from the definition
 of multiary quasigroups.
 Let us prove the first one. Any tuple
 $(x_0, \ldots, x_{n-2})$ satisfying~\eqref{eq:xmu}
 has the form
 \begin{multline*}
  (x_0, \ldots, x_{n-2}) = 
 (
 x_0,x_0+z_0,x_0-z_0+\mu_0,\\
 x_3,x_3+z_1,x_3-z_1+\mu_1,\\
  \ldots,\\
 x_{n-4},x_{n-4}+z_{n'-1},x_{n-4}-z_{n'-1}+\mu_{n'-1}
  ), 
 \end{multline*}
  where $z_i=x_{3i+1}-x_{3i}$, $i=0, \ldots , n'-1$. 
  So, 
  $$
  \lambda_{\vc\mu}(x_0, \ldots, x_{n-2}) = 
  \lambda''_{\vc\mu}(x_0,x_3,\ldots, x_{n-4},z_0,z_1,\ldots,z_{n'-1})
  $$
  for some function $\lambda''_{\vc\mu}$. Let us show that 
  $\lambda''_{\vc\mu}$ does not depend on $x_0$, $x_3$, \ldots, $x_{n-4}$.
  If $\vc{x}=(x_0, \ldots, x_{n-2})$ satisfies~\eqref{eq:xmu},
  then 
  $\vc{x}+\vc{e}_{012}$, $\vc{x}+\vc{e}_{021}$, 
  and 
  $\vc{x}+\vc{e}_{111}$, 
  where 
  $\vc{e}_{ijk}=(i,j,k,0, \ldots ,0)$, 
  also satisfy~\eqref{eq:xmu}.
  Since $\vc{x}$, $\vc{x}+\vc{e}_{012}$, and $\vc{x}+\vc{e}_{021}$ 
  are at mutual distance~$2$ from each other,
  we see from~\eqref{eq:la} that 
  $\{ \lambda_{\vc\mu}(\vc{x}), 
      \lambda_{\vc\mu}(\vc{x}+\vc{e}_{012}), 
      \lambda_{\vc\mu}(\vc{x}+\vc{e}_{021}) \}
      = \{ 0,1,2 \}$.
Similarly, 
  $\{ \lambda_{\vc\mu}(\vc{x}+\vc{e}_{111}), 
      \lambda_{\vc\mu}(\vc{x}+\vc{e}_{012}), 
      \lambda_{\vc\mu}(\vc{x}+\vc{e}_{021}) \}
      = \{ 0,1,2 \}$. 
Therefore, 
$\lambda_{\vc\mu}(\vc{x})=\lambda_{\vc\mu}(\vc{x}+\vc{e}_{111})$
and, in particular,
$\lambda''_{\vc\mu}(x_0,x_3,\ldots, x_{n-4},z_0,z_1,\ldots,z_{n'-1})$
does not depend on~$x_0$.
Similarly, it does not depend on $x_3$, \ldots, $x_{n-4}$,
and 
 $$ \lambda''_{\vc\mu}(x_0,x_3,\ldots, x_{n-4},z_0,z_1,\ldots,z_{n'-1})
 = \lambda'_{\vc\mu}(z_0,z_1,\ldots,z_{n'-1})$$
 for some $\lambda'_{\vc\mu}$, which is an $n'$-ary quasigroup, by the definition.
\end{proof}

Summarizing Lemmas~\ref{l:HK} and~\ref{l:q3}, we obtain the following.

\begin{theorem} \label{th:q3r1}
 Let $C$ be a $3$-ary $1$-perfect code of length
 $n=(3^m-1)/2$ of rank at most~$+1$ and 
 $C^\star$ be the $3$-ary Hamming code
 of length 
 $n'=(3^{m-1}-1)/2$.
 Then for some automorphism~$\psi$ of~$H(n,3)$,
 it holds
 $$\psi(C) = \bigcup_{\vc{\mu}\in C^\star} K_{\vc\mu},$$
 where
 \begin{multline*}
 K_{\vc\mu} =
 \big\{(x_0,x_1,\ldots,x_{n-1}): \
 x_2=\mu_0-x_0-x_1,\ 
 x_5=\mu_1-x_3-x_4,\ \ldots,\\
 x_{n-2}=\mu_{(n-4)/3}-x_{n-4}-x_{n-3},\\
 x_{n-1} = \lambda_{\vc\mu}(x_1-x_0,x_4-x_3,\ldots,x_{n-3}-x_{n-4}) \big\}
 \end{multline*} 
 for some $(n-1)/3$-ary 
 quasigroup~$\lambda_{\vc\mu}$ of order~$3$,
 $\vc\mu\in C^\star$.
\end{theorem}
In contrast to multiary quasigroups of higher orders,
all $t$-ary quasigroups of order~$3$ 
(and the corresponding $3$-ary distance-$2$ MDS codes)
are affine:
\begin{proposition}[{\cite[Corollary~13.25, Exercise~13.11]{LaywineMullen}}]\label{p:3aff}
 There are exactly $2\cdot 3^t$ $t$-ary quasigroups of order~$3$.
 Each of them has the form
 \begin{equation}\label{eq:3aff}
   f(x_0,\ldots,x_{t-1}) = a_0x_0 + \ldots + a_{t-1}x_{t-1}+a
 \end{equation}
 for some $a_0$, \ldots, $a_{t-1}$ from~$\{1,2\}$ 
 and $a$ from~$\{0,1,2\}$.
 \end{proposition} 
 So, the characterization
 of $3$-ary $1$-perfect codes of rank at most~$+1$
 in Theorem~\ref{th:q3r1} is constructive.

\begin{corollary}\label{c:r-k}
 The dimension of the kernel of a $3$-ary $1$-perfect code~$C$
 of length~$n$ and rank~$+1$
 is at least $(n-1)/3$.
\end{corollary}
\begin{proof}
Without loss of generality, we can assume that
the conclusion of Theorem~\ref{th:q3r1}
holds with the identity~$\psi$.
From the proof of Lemma~\ref{l:q3}, we
see that $(1,1,1,0, \ldots, 0)$ is in the kernel
of each $K_{\vc\mu}$, $\vc\mu \in C^*$
and hence belongs to the kernel of~$C$.
Similarly, the kernel contains
$(0,0,0,1,1,1,0, \ldots, 0)$, \ldots,  $(0, \ldots, 0,1,1,1,0)$.
\end{proof}
As we will see in Section~\ref{s:class} (Table~\ref{t:rk}),
the bound is tight for $n=13$: there are 
concatenated $(13,3^{10},3)_3$-codes
of rank~$+1$ with kernel of size~$3^4$.
The smallest automorphism group of such concatenated codes,
however, is twice larger,~$162$.
We have analyzed $10000$ random $(13,3^{10},3)_3$-codes of rank~$+1$
and have not found a code with
automorphism group of order~$81$ (or of any other order that does not occur among concatenated codes of rank~$+1$).
The most typical values of the order are 
$162$ ($52.2\%$ of the cases), 
$243$ ($29.7\%$), 
$486$ ($14.8\%$), 
and~$729$ ($2.8\%$)
(the dimension of the kernel is 
$4$ in $ 95.6\% $,
$5$ in $ 4.3\% $,
$6$ in $ 0.05\% $ 
of the cases).
We conjecture that
this is a small-length phenomena,
and as $n$ grows, 
almost all $3$-ary $1$-perfect 
codes of length~$n$ and rank~$+1$
have exactly~$3^{\frac{n-1}3}$ automorphisms.

\subsection{The number of rank +1 codes}\label{ss:no}

\begin{lemma}\label{l:fixed}
In $H(n=\frac{3^m-1}2,3)$,
 the number of $1$-perfect codes 
 of rank $+1$ with the same affine span is
 \begin{equation}\label{eq:N'}
  N'(n)= (3\cdot 2^{\frac{n-1}3})^{3^{\frac{n-1}3 - m+1}} 
 - {6^{\frac{n-1}3}}\cdot{3^{-m+2}}.
 \end{equation}
\end{lemma}
\begin{proof}
 Assume without loss of generality that 
 one of the $1$-perfect codes of rank~$+1$
 satisfies the conclusion
 of Theorem~\ref{th:q3r1} with the identity~$\psi$.
 Denote by~$S$ its affine span; since $\psi$ is identity,
 $S$ is linear.
 Then the other $1$-perfect codes of rank~$+1$
 with affine span~$S$
 also satisfy the conclusion
 of Theorem~\ref{th:q3r1} with the identity~$\psi$
 and the same $C^\star$. The number of codes that
 satisfy the conclusion
 of Theorem~\ref{th:q3r1} with the identity $\psi$
 is $Q_{(n-1)/3}\cdot |C^{\star}|$, where $Q_t=3\cdot 2^{t}$
 is the number of $t$-ary quasigroups of order~$3$ 
 (Proposition~\ref{p:3aff})
 and $|C^{\star}|=3^{(n-1)/3 - m+1}$.
 However, some of these codes are of rank~$+0$,
 and it remains to find the number of such codes, 
 subsets of~$S$ that are cosets of Hamming codes.
 
 We first count the number of Hamming codes, 
 subsets of~$S$.
 Let $H^*$ be a check matrix
 of the Hamming code~$C^*$. 
 It consists of~$(n-1)/3$ 
 mutually non-colinear
 columns of height~$m-1$. It is straightforward
 that a check matrix~$H$ of~$S$ can be constructed 
 by repeating each column of~$H^*$ three times
 and adding one all-zero column.
 To complete~$H$ to a check matrix of a Hamming code,
 we need to add one row that makes all columns
 mutually non-colinear. There are $6^{(n-1)/3}$
 of ways to do so, assuming without loss of generality 
 that the last symbol is~$1$.
 Since adding to the last row 
 a linear combination of the first $m-1$ rows
 does not change the linear span of the rows,
 we have $6^{(n-1)/3}/3^{m-1}$ different 
 Hamming subcodes of $S$. Each of them has $3$ cosets
 in $S$, so the total number of cosets is
 $6^{(n-1)/3}/3^{m-2}$.
\end{proof}

\begin{theorem}\label{th:all}
The number of $1$-perfect codes 
 of rank $+1$ in $H\big(n=\frac{(3^m-1)}2,3\big)$ is
 $$
 \frac{ n!\cdot 6^n }{ 
 |\GL[3]{m{-}1}| \cdot 6^{\frac{n-1}3}
   \cdot 3^{n-m+1} }
   \cdot N'(n),
 $$
 where $N'(n)$ is from~\eqref{eq:N'}.
\end{theorem} 
 \begin{proof}
  The total number is~$N'(n)$ multiplied 
  by the number of sets equivalent to the subspace~$S$
  (we keep the notation from Lemma~\ref{l:fixed}
  and its proof).
  The group of monomial automorphisms of~$S$
  has order 
  $$
  |\MAut(S)|=|\MAut(C^*)|\cdot 6^{(n-1)/3}
  = |\GL[3]{m{-}1}|\cdot 6^{(n-1)/3}
  ,$$
  where $6=3!$ is the number of permutations 
  of three coordinates that correspond to three equal
  columns of~$H$.
  Hence, $|\Aut(S)|=|\MAut(S)|\cdot|S|$,
  and the number of sets (affine spaces)
  that are equivalent to~$S$ is
  
  \smallskip
  \mbox{}\hfill
  $
  \displaystyle
  \frac{|\Aut(H(n,3))|}{|\Aut(S)|} = 
  \frac{n!\cdot 6^n }{ 
   |\GL[3]{m{-}1}| \cdot 6^{(n-1)/3}
   \cdot 3^{n-m+1} }.
   $
 \end{proof}
 
 \begin{corollary}\label{c:noneq}
  The number of equivalence classes
  of $1$-perfect codes of rank $+1$ in $H(n=\frac{(3^m-1)}2,3)$ 
  is not less than
 $$
 \bigg \lceil
 \frac{ N'(n)}{ 
 |\GL[3]{m{-}1}| \cdot 2^{\frac{n-1}3}
   \cdot 3^{n-m+1} }
   \bigg \rceil.
 $$
 \end{corollary}
\begin{proof}
 The number of equivalence classes is not less than
 the value of all codes from Theorem~\ref{th:all}
 divided by the maximum cardinality of an equivalence class.
 The maximum cardinality of an equivalence class 
 equals $|\Aut(H(n,3))|$, i.e., $n!\cdot 6^n$,
 divided by the minimum order 
 of the automorphism group of a code from the considered family.
 The minimum order 
 of the automorphism group of a ternary $1$-perfect code of rank~$+1$
 is not less than~$3^{\frac{n-1}3}$, by Corollary~\ref{c:r-k}.
\end{proof}

For example, for $m=3$, we have $ N'(n) = 1352605460594256 $,
the total number of $(13,3^{10},3)_3$ codes of
rank $11$ is $ 9982462029409199967436800 $, the number of equivalence
classes is at least $ 9942054 $.
Based on experiments with random codes
mentioned in the end of Section~\ref{s:char},
we expect that
the real number of equivalence
classes is more than $20$ millions
(and only $1164330$ of them
can be obtained by concatenation, see Table~\ref{t:rk} in Section~\ref{s:class}).

\subsection{Switchings}\label{ss:swi}
In this section, we will show that the ternary $1$-perfect codes of rank at most~$+1$
can be obtained from each other by a sequence of two-coordinate switchings.
A similar result for extended binary $1$-perfect codes of rank at most~$+2$
was proved in~\cite{KroPot:swi}. 

Assume that we have two $q$-ary $1$-perfect codes~$C$, $C'$ 
of length~$n$ and
an automorphism $\beta=(\pi,\vc\theta)$ of $H(n,q)$ 
such that the coordinate
permutation~$\pi$ fixes all coordinates 
except maybe the $i$th and the $j$th coordinates
and the isotopy~$\vc\theta$ fixes the values of all coordinates 
except maybe the $i$th and the $j$th coordinates.
We say that $C'$ is a \emph{two-coordinate switching} of~$C$,
or an \emph{$\{i,j\}$-switching} of~$C$, or, more concrete, 
a \emph{$\beta$-switching} of~$C$, if
$$ C' \subset C \cup \beta(C). $$
(Similarly, three-, four-, etc. coordinate switchings
can be defined.) 
For a given~$\beta$, the process of finding all switchings of~$C$
is rather simple. We construct the \emph{inconsistency}
bipartite graph $G_{1,2}(C \cup \beta(C))$ on the vertex set $C \cup \beta(C)$, 
where two words are adjacent if
the distance between them is~$1$ or~$2$.
We collect in~$C'$ all isolated vertices of $G_{1,2}(C \cup \beta(C))$
and add a bipartite part of the remaining subgraph.
If this subgraph has more than one connected components,
then a  bipartite part can be chosen in more than two ways
and there are switchings different from~$C$ and~$\beta(C)$.
The process of finding a new code~$C'$ from~$C$ as described above 
is also called \emph{switching}.

\begin{theorem}\label{th:swi}
 The set of codes of rank at most~$+1$ is connected with respect to
the two-coordinate switching.
\end{theorem}

The proof is more or less straightforward
from the corollary of the following lemma.

\begin{lemma}\label{l:qua-swi}
 Every two $t$-ary quasigroups $f$, $f'$ of order~$3$ can be obtained
 from each other by a sequence of transformations 
 $\gamma_{i,a}$, $i\in\{0,\ldots,t-1\}$, $a\in\{0,1,2\}$,
 where $\gamma_{i,a}$ swaps the values of $\{0,1,2\} \backslash \{a\}$
 in the $i$th argument of the function.
\end{lemma}
\begin{proof}
By Proposition~\ref{p:3aff}, every  
 $t$-ary quasigroup $f$ or order~$3$ can be written in the worm~\eqref{eq:3aff}, where
 $a_0,\ldots,a_{t-1}\in\{1,2\}$, $a\in \{0,1,2\}$. 
 To change $a_i$, we can apply $\gamma_{i,0}$. 
 To change $a$, we can apply $\gamma_{0,0}\gamma_{0,1}$.
\end{proof}

\begin{corollary}\label{c:comp-swi}
 Every two different $\vc\mu$-components~$K_{\vc\mu}$ and~$K'_{\vc\mu}$ 
 satisfying, for a given~$\mu$,
 the conclusion of Theorem~\ref{th:q3r1}
 are obtained from each other by a sequence of two-coordinate
 switchings.
\end{corollary}
\begin{proof}
By Theorem~\ref{th:q3r1}, the $\vc\mu$-components~$K_{\vc\mu}$ and~$K'_{\vc\mu}$
are constructed from some $(n-1)/3$-ary quasigroups~$\lambda$ and~$\lambda'$
of order~$3$.
By Lemma~\ref{l:qua-swi}, it is sufficient to prove the claim
for two quasigroups that are obtained from each other by the
transformation~$\gamma_{i,a}$, for some $i\in \{0,\ldots,(n-4)/3\}$
and $a\in\{0,1,2\}$. Without loss of generality, assume $i=0$.
Consider two subcases.

\emph{Subcase $a=0$.} An arbitrary word from~$K_{\vc\mu}$ has the form
\begin{multline*}
  \big(x_0,\ x_0+z_0,\ x_0-z_0+\mu_0, \quad
     x_3,\ x_3+z_1,\ x_3-z_1+\mu_1, \quad\ldots, \\
      x_{n-4},\ x_{n-4}+z_{(n-4)/3},\ x_{n-4}-z_{(n-4)/3}+\mu_{(n-4)/3},
\quad\lambda(z_0,\ldots,z_{(n-4)/3})\big). 
\end{multline*}
After transforming $(z_0,\ldots,z_{(n-4)/3})$ with~$\gamma_{0,0}$,
the value of the $1$st coordinate turns from $x_0+z_0$ to $x_0-z_0$,
and the value of the  $2$nd coordinate turns from $x_0-z_0+\mu_0$ to $x_0+z_0+\mu_0$.
This is the same as permuting these two coordinates
and adding $(0,-\mu_0,\mu_0,0,\ldots,0)$, which is a 
$\{1,2\}$-switching by the definition.

\emph{Subcase $a=1$ (similarly, $a=2$).}
After transforming $(z_0,\ldots,z_{(n-4)/3})$ 
with~$\gamma_{0,1}$,
the value of the $1$st coordinate turns 
from $x_0+z_0$ to $x_0-z_0+2$,
and the value of the  $2$nd coordinate turns 
from $x_0-z_0$ to $x_0+z_0+1$.
This is the same as permuting these two coordinates 
and adding $(0,2-\mu_0,1+\mu_0,0,\ldots,0)$,
which is again a $\{1,2\}$-switching.
\end{proof}

\begin{proof}[Proof of Theorem~\ref{th:swi}]
 Utilizing the characterization in Theorem~\ref{th:q3r1},
 we see that
 for the identity~$\psi$ the claim follows 
 from Corollary~\ref{c:comp-swi}.
 It remains to observe that the action of an arbitrary~$\psi$
 can be represented as a sequence of two-coordinate switchings.
\end{proof}

\subsection{Maximum kernel}\label{s:ker}

In this section, motivated by a question of one of the reviewers,
we consider the structure of a nonlinear ternary $1$-perfect code
with maximum kernel dimension.
The following theorem considers only length-$13$ codes; 
however, the most part of the proof (except the last paragraph)
is applicable to an arbitrary ternary $1$-perfect 
code~$C$ with kernel of size~$|C|/3$.

\begin{theorem}\label{th:mxker}
 There is only one equivalence class of $1$-perfect ternary codes 
 of length~$13$ with kernel of dimension~$9$.
\end{theorem}
\begin{proof}
Let $C = K \cup (\vc a+K) \cup (\vc b+K)$ be a ternary $1$-perfect 
code with kernel~$K$. Since $C$ is nonlinear, its rank
is $\dim(K)+2$, i.e., $+1$.

We claim that $C' = K \cup (\vc a+K) \cup (2\vc a+K)$ 
is also a $1$-perfect code. It is sufficient to show that there are no
two codewords $\vc x$ in $K \cup (\vc a+K)$ and  
$\vc y$ in $(2\vc a+K)$ at distance less than~$3$ from each other.
If $\vc x \in K$, then $\vc x,\vc y \in 2\vc a + C$;
if $\vc x \in (\vc a+K)$, then $\vc x,\vc y \in \vc a + C$.
In both cases, $\vc x$ and~$\vc y$ belong to the same $1$-perfect code
(and, moreover, to different cosets of its kernel),
and hence the distance between them is at least~$3$.

Now, we see that $C$ and $C'$ are $1$-perfect codes with symmetric
difference $(\vc b+K) \cup (2\vc a+K)$.
Moreover, $C'$ is linear.
By the definition of a $1$-perfect code, 
every word from $(\vc b+K)$ 
is at distance~$1$ from $(2\vc a+K)$.
It follows that $(\vc b+K) = \vc e+(2\vc a+K)$
for some weight-$1$ word $\vc e$.

We summarize: the code~$C$ is obtained from some linear
$1$-perfect code~$C'$ by translating an affine subspace of size~$|C'|/3$ with a translation vector of weight~$1$.

Since all linear $1$-perfect codes are equivalent,
we can assume without loss of generality that $C'$
has the form from the conclusion 
of Theorem~\ref{th:q3r1}
with identity~$\psi$ and 
$\lambda_{\vc\mu}(y_0,y_1,\ldots)=y_0+y_1+\ldots$
for all~$\vc\mu$. Moreover, we can assume 
that the nonzero value of~$\vc e$ (say, $e$) 
is in the last coordinate, i.e., $\vc e = (0,\ldots,0,e)$
(here we utilize the well-known fact that for every two coordinates~$i$
and~$j$ there is an automorphism 
of the Hamming code that sends~$i$ to~$j$).

After translating the subset $(2\vc a+K)$
in direction~$\vc e$, we see that the resulting code~$C$
still satisfies the conclusion of Lemma~\ref{l:HK}
and hence the conclusion of Theorem~\ref{th:q3r1}
with identity~$\psi$.
The only difference with~$C'$ is that for~$C$,
we have
$\lambda_{\vc\mu}(y_0,y_1,\ldots)=e+y_0+y_1+\ldots$
for $\mu\in K^*$, where~$K^*$ is some affine subset 
of~$C^*$ of cardinality~$|C^*|/3$.

It remains to observe that all affine $3$-subsets of the 
Hamming $(4,9,3)_3$ code~$C^*$ are equivalent and that
the subcases $e=1$ and $e=2$ lead to equivalent codes:
$$
K \cup (\vc a+K) \cup (2\vc e+2\vc a+K) 
= \vc a + 2\big(K \cup (\vc a+K) \cup (\vc e+2\vc a+K)\big).
$$
\end{proof}


\section{Enumeration of concatenated ternary 1-perfect codes 
of length~13}\label{s:class}
In this section, we describe the computer-aided classification of
the concatenated ternary $1$-perfect codes of length~$13$.
As intermediate steps, of independent interest,
we get the classification of RM-like $(9,3^6,3)_3$ codes, 
collections of such codes (including RM-like partitions) 
that are subsets of the 
all-parity-check code~$M$,
and partitions of $\FF_3^4$ into $1$-perfect codes.

Before we describe the approaches we use on each step of the classification,
we briefly discuss recognizing the equivalence, 
which is a very important and the most time-con\-su\-ming tool
for such classifications.
\subsection{Equivalence and graph isomorphism}\label{ss:iso}
A usual way to work with the equivalence of codes is 
to represent them by graphs in such a way that
two codes are equivalent if and only if the corresponding  
graphs are isomorphic, 
see~[12, \S3.3.2]. 
It is easy to adopt such an approach 
for collections of codes; 
one of the ways
is to represent a collection~$(C_i)_{i=0}^{k-1}$
of codes in~$\FF_q^n$ as a mixed-alphabet code in 
$\FF_q^n\times\{0,\ldots,k-1\}$:
$$ C = \left\{(c_0,\ldots,c_{n-1},i):\ i\in\{0,\ldots,k-1\},\ (c_0,\ldots,c_{n-1})\in C_i \right\}. $$
A standard software that helps to recognize the graph isomorphism is 
\texttt{nauty\&traces}~[13]; 
it is realized as a package that can be used in \texttt{c} or \texttt{c++}
programs. With this package, for a graph one can compute its 
canonically-labeled version, such that two graphs are isomorphic
if and only if the corresponding canonically-labeled graphs are equal to each other.
The same procedure computes the automorphism group of the graph,
which can be used for the numerical validation of the results.
The library suggests two alternatives 
of such procedure, \texttt{nauty} and \texttt{traces}.
According to our experience,
\texttt{traces} worked faster on codes with considered parameters.

\subsection{Classification of RM-like codes of length~9}\label{ss:RM}
Every RM-like code is a subset of a distance-$2$ MDS code, say~$M$.
By Proposition~\ref{p:3aff}, 
in the ternary case such code~$M$ is unique 
up to equivalence, 
for each length,
and we assume without loss of generality that 
\begin{equation}\label{eq:M}
  M = \left \{ (x_0,\ldots,x_{8})\in\mathbb{F}_3^{9} : \ x_0+\ldots+x_{8}=0\right\}.
\end{equation}
By~$\overline M$, we denote the complement of~$M$; i.e.,
$$
\overline M = \left\{ (x_0,\ldots,x_{8})\in\mathbb{F}_3^{9} :
\ x_0+\ldots+x_{8}\in\{1,2\}\right\}. 
$$
Our goal, at this stage, is to classify all RM-code subsets of~$M$ up to equivalence.
Without loss of generality, we take~$\vc{0}$ as a codeword.
We say that a set~$C_i$ of vertices of~$H(9,3)$ is a \emph{partial code} of level~$i$ if
\begin{enumerate}
 \item[(I)] $C_i$ contains~$\vc{0}$,
 \item[(II)] $C_i$ consists of words of weight at most~$i$,
 \item[(III)] every word of weight at most~$i-1$ in~$\overline M$
 is adjacent to exactly one codeword of~$C_i$, and
 \item[(IV)] $C_i$ is a distance-$3$ code.
\end{enumerate}

The classification algorithm we use is based on the straightforward fact 
that by removing the weight-$(i+1)$ codewords from
a {partial code} of level~$i+1$ 
we obtain a {partial code} of level~$i$.

1. We start with the singleton~$\{\vc0\}$, which is a unique \emph{partial code} of level~$1$ and~$2$.

2. Assume that at step~$i$ we have found representatives 
of all equivalence classes of partial codes of level~$i$. 
For each representative~$C_i$,
we can find all partial codes of level~$i+1$ 
that include~$C_i$ in the following way.
\begin{itemize}
 \item Denote by~$W_i$ the set of weight-$i$ words in $\overline M$ that are at distance more than~$1$ 
 from~$C_i$.
 \item Denote by~$R_i$ the set of weight-$(i+1)$ words in~$M$ that are at distance at least~$3$
 from~$C_i$.
 \item Let $X=(X_{\vc{x},\vc{y}})$ be the $\{0,1\}$-matrix whose rows are indexed by elements of~$R_i$
 and columns are indexed by elements of $W_i$ such that ~$\vc{x}$ from~$R_i$ and~$\vc{y}$ from~$W_i$
 are adjacent if and only if $X_{\vc{x},\vc{y}}=1$.
 \item As follows from (III), 
 for every partial code~$C_{i+1}$ of level~$i+1$, the sum of rows of~$X$ indexed by the elements of
  $C_{i+1}\backslash C_{i}$ equals the all-one row. Finding such collections of rows for a $\{0,1\}$-matrix
  is an instance of the well-known exact cover problem, 
  which is usually solved by Donald Knuth's Algorithm X (already realized as a function in many programming languages).
  \item After completing~$C_{i}$ by a solution of the exact cover problem with matrix~$X$, we need to check 
  property (IV); all solutions that satisfy it correspond to partial codes of level~$i+1$, by the definition.
\end{itemize}
Finally, all found continuations are checked for equivalence, and we keep only nonequivalent representatives.

We repeat p.2 for $i=3,4,5,6,7,8,9$ and obtain the following results:

There are $705600$ partial codes of level~$3$; they form $9$ equivalence classes;
the orders of the automorphism groups are 
$864$, $108$, $32$, $24$, $18$, $6$, $6$, $4$, $4$.

\begin{remark}
It is not difficult to observe that every partial code of level~$3$ consists
of the all-zero word, twelve words with three~$1$s, and
twelve words with three~$2$s. The twelve words
from each of the last two groups (to be exact, 
the sets of indices of nonzero coordinates of these words)
form a combinatorial structure known as a Steiner triple system of order~$9$,
SQS$(9)$, see e.g.~\cite{ColMat:Steiner}. 
There are $840$ different SQS$(9)$, and all of them are isomorphic.
There are $840^2=705600$ pairs of SQS$(9)$, 
and $9$ isomorphism classes of such pairs.
\end{remark}

The partial codes of level~$3$ are continued to, respectively, $4$, $4$, $0$, $4$, $0$, $0$, $0$, $0$, $0$  
nonequivalent partial codes of level~$4$, with automorphism group orders 
$864$, $216$, $72$, $72$, $108$, $108$, $36$, $36$, $12$, $12$, $12$, $12$.
Six of these codes, with automorphism group orders $864$, $72$, $108$, $36$, $12$, $12$,
are continued to a partial code of level $5$, $6$, $7$, $8$, and $9$;
at each step
the continuation is unique and preserves the automorphism group of the ``parent'' partial code.
The other $6$ partial codes of level~$4$ are continued uniquely to partial codes of level~$5$, 
but not to partial codes of level~$6$.

\begin{theorem}[computational]\label{th:RMlk}
  There are $1428840$ 
 $(9,3^6,3)_3$ RM-like codes that are subcodes of~$M$.
 $158760$ of them contain the all-zero word; 
 they form $4$ equivalence classes (the corresponding automorphism group orders are $629856$, $78732$, $8748$, $5832$), 
 $6$ monomial equivalence classes 
 (the corresponding monomial automorphism group orders are $864$, $108$, $12$, $72$, $36$, $12$),
 $7$ permutation  equivalence classes
 (the corresponding permutation automorphism group orders are $432$, $54$, $6$, $36$, $18$, $12$, $12$).
\end{theorem}

\subsection{Collections of disjoint RM-like codes}\label{ss:RM9}
For the concatenation construction, we need a partition
of the distance-$2$ MDS code~$M$~\eqref{eq:M}
into $9$ RM-like codes.
In this section, we consider the classifications
of collections of disjoint RM-like subcodes of~$M$;
we call such a collection a \emph{$k$-collection}, 
where $k$ is the number of codes in it.
We classify them recursively.
The algorithm is rather straightforward; 
however, the amount of calculations was huge (it took about 30 core-years to finish it),
and the details considered below were essential to make it doable with reasonable computational
resources.

We first define the equivalence for $k$-collections. 
Two collections $(C_i)_{i=0}^{k-1}$ and $(D_i)_{i=0}^{k-1}$ 
of vertex sets of $H(n,q)$ are \emph{equivalent}
if there is an automorphism $\gamma$ of the graph $H(n,q)$
and a permutation~$\tau$
of $\{0,\ldots,{k-1}\}$ such that $\gamma(C_i)=D_{\tau(i)}$, 
$i=0,\ldots,{k-1}$.
If, additionally, $\tau({k-1})={k-1}$,
then we will say that
$(C_i)_{i=0}^{k-1}$ and $(D_i)_{i=0}^{k-1}$
are \emph{strongly equivalent}.
The set or all pairs $(\gamma,\tau)$ such that
$\gamma(C_i)=D_{\tau(i)}$, $i=0,\ldots,{k-1}$,
forms the \emph{automorphism group} 
of $(C_i)_{i=0}^{k-1}$.

From Section~\ref{ss:RM} we know that the number of equivalence classes of RM-like sub-codes of~$M$
is~$4$. For every $k$-collection $(C_i)_{i=0}^{k-1}$, we define its \emph{type} as the sequence
$t_0t_1\ldots t_{k-1}$ where $C_i$ belongs to the $t_i$th equivalence class, 
$t_i\in\{0,1,2,3\}$, $i=0,1,\ldots,k-1$.
The type is \emph{sorted} if $t_0\le t_1\le \ldots \le t_{k-1}$.
Obviously, every equivalence class of $k$-collections
has a representative of sorted type,
and this sorted type is uniquely defined for the class.
Trivially, removing the last code from a $k$-collection 
$(C_i)_{i=0}^{k-1}$, $k>1$,
we obtain a $(k-1)$-collection $(C_i)_{i=0}^{k-2}$; 
thus, we will say that 
$(C_i)_{i=0}^{k-1}$ is 
a \emph{continuation} of $(C_i)_{i=0}^{k-2}$.
It is also clear that if two $k$-collections
are strongly equivalent, then they are continuations
of equivalent $(k-1)$-collections.

For a given sorted type $t_0t_1\ldots t_{k-1}$, 
we classify all $k$-collections of this type
up to equivalence in two steps.

\begin{itemize}
 \item{(I)} At first, for each representative $(C_i)_{i=0}^{k-2}$
       of $(k-1)$-collections 
       of type $t_0\ldots t_{k-2}$, we construct all possible 
       continuations of type $t_0\ldots t_{k-2}t_{k-1}$. For this,
       we consider all RM sub-codes of $M$ from the $t_{k-1}$th equivalence class
       (there are $7560$, $60480$, $544320$, and $816480$ such codes 
        for $t_{k-1}=0,1,2,3$, respectively).
        Those codes who are disjoint with all $C_i$, $i=0,\ldots,k-2$, 
        are used for the role of $C_{k-1}$ to form a continuation~$(C_i)_{i=0}^{k-1}$
        of~$(C_i)_{i=0}^{k-2}$.
        The resulting $k$-collections are checked for the strong equivalence,
        and we keep only representatives of strong equivalence classes.
        This step can be done separately for each initial $(k-1)$-collection,
        which allows to process different $(k-1)$-collections on different 
        machines with relatively small (several gigabytes) amount of memory.
\item{(II)} Next, all representatives  of strong equivalence classes kept at step (I)
 for all different initial $(k-1)$-collections of the same type are checked for equivalence and 
 representatives of equivalence classes are collected. Because of the huge amount of resulting
 representatives, this step is processed on a machine with large amount of memory (more than 160~Gb).
 One of benefits of the two-step approach, apart from the rational use of computational resources,
 is that comparing for equivalence, especially for big values of $k$ ($6$--$9$), 
 takes much more time than comparing for strong equivalence, 
 and the precalculation made at step (I)
 minimizes the amount of such operations.
\end{itemize}

 $1$-collections are essentially RM-like codes, 
 which are classified in Section~\ref{ss:RM9}.
 With the two-step algorithm described above, 
 nonequivalent $k$-collections are classified
 subsequently for $k=2,\ldots,9$.
 \begin{theorem}[computational]\label{th:RM9}
  There are 
$4$, $131$, $10956$, $118388$, $501915$, $945965$, $755066$, $314833$, and $65436$  
  equivalence classes of $k$-collections of disjoint 
  RM-like subcodes of the distance-$2$ MDS code~$M$~\eqref{eq:M} for
  $k=1$, $2$, $3$, $4$, $5$, $6$, $7$, $8$, $9$, respectively.
  The distribution of equivalence classes of $9$-collections 
  (RM-like partitions of~$M$) by type is the following:\newline
  \mbox{}\!\!\!\!\!
\begin{tabular}{l@{\ }l@{\ }l@{\ }l@{\ }l}
 $000000000{:}\ 6$,     &
$000000011{:}\ 11$,     &
$000000111{:}\ 6$,      &
$000000222{:}\ 20$,     &
$000000333{:}\ 41$,     \\
$000001111{:}\ 26$,     &
$000011111{:}\ 11$,     &
$000011222{:}\ 107$,    &
$000011333{:}\ 173$,    &
$000111111{:}\ 66$,     \\
$000111222{:}\ 41$,     &
$000111333{:}\ 70$,     &
$000222222{:}\ 347$,    &
$000222333{:}\ 990$,    &
$000333333{:}\ 885$,    \\
$001111111{:}\ 24$,     &
$001111222{:}\ 199$,    &
$001111333{:}\ 381$,    &
$011111111{:}\ 51$,     &
$011111222{:}\ 112$,    \\
$011111333{:}\ 208$,    &
$011222222{:}\ 1205$,   &
$011222333{:}\ 3493$,   &
$011333333{:}\ 3006$,   &
$111111111{:}\ 26$,     \\
$111111222{:}\ 99$,     &
$111111333{:}\ 237$,    &
$111222222{:}\ 381$,    &
$111222333{:}\ 1180$,   &
$111333333{:}\ 1126$,   \\
$222222222{:}\ 3228$,   &
$222222333{:}\ 14356$,  &
$222333333{:}\ 21405$,  &
$333333333{:}\ 11919$   &
\end{tabular}
(we skip the sorted types that are not represented, e.g.,
$000000022{:}~0$).
 \end{theorem}

\subsection{1-perfect partitions of length 4}\label{ss:p9}
There are $72$ $1$-perfect $(4,9,3)_3$ codes;
all of them are equivalent to the $3$-ary Hamming code 
of length~$4$.
Straightforward computations show that from these $72$ codes,
one can choose $9$ pairwise disjoint codes in $104$ ways.
\begin{theorem}[computational]\label{th:p4}
 There are exactly two equivalence classes of partitions of $\FF_3^4$ into $1$-perfect codes.
 Each of the $8$ partitions from the smallest class consists 
 of the cosets of the same Hamming code, 
 and the order of its automorphism group is~$384$.
 Each of the $96$ remaining partitions consists of cosets 
 of two different Hamming codes, in the quantity of~$6$ and~$3$,
  and the automorphism group order is~$32$.
\end{theorem}

\subsection{Concatenated codes}\label{ss:cc}

In this section, we describe the final steps of the classification
of concatenated $3$-ary $1$-perfect codes of length~$13$.
As reported in Sections~\ref{ss:RM9} and~\ref{ss:p9},
we have classified up to equivalence 
the partitions of the distance-$2$ MDS code~$M$ 
into RM-like $(9,3^6,3)_3$ codes and the partitions
of $\FF_3^4$ into $1$-perfect $(4,9,3)_3$ codes.
The third ingredient of the concatenation construction is a permutation
of $9$ codes. There are $9!=362880$ different permutation, which form
the symmetric group $\mathrm{Sym}(9)$, 
and this~$9!$ is the number of different concatenated codes 
that can be obtained
from given partitions of~$M$ and~$\FF_3^4$.
In the following subsection,
using the knowledge about the automorphism groups of the used
RM-like and $1$-perfect partitions, 
we certify that some of these codes
are guaranteedly equivalent; this essentially reduces
the number of considered codes.

\subsubsection{Double-cosets}\label{sss:dcoset}
The following fact is well known; in particular,
similar arguments were used in~\cite{Phelps:2000}
for the classification
of concatenated binary codes.

\begin{lemma}\label{l:dcoset}
Assume that 
  $\bar C = (C_0,\ldots,C_{k-1})$ and 
  $\bar P = (P_0,\ldots,P_{k-1})$ 
  are collections of mutually disjoint
  codes in~$H(n',q)$ and~$H(n'',q)$, respectively.
  Assume that $\alpha$ is a permutation of $\{0,\ldots,{k-1}\}$
  and we have two automorphisms
  $(\pi',\vc\theta',\tau')$
  and $(\pi'',\vc\theta'',\tau'')$
  of~$\bar C$
  and~$\bar P$ respectively.
  Then the concatenated codes 
  $$ \bigcup_{i=0}^{k-1} C_i P_{\alpha(i)}
  \qquad
  \mbox{and}
  \quad
  \bigcup_{i=0}^{k-1} 
  C_i P_{\tau''(\alpha(\tau'(i)))}
  $$
  are equivalent.
 \end{lemma}
\begin{proof}
   \begin{multline*}
  \bigcup_{i=0}^{k-1} 
  C_i \times P_{\tau''(\alpha(\tau'(i)))}    
  \stackrel{j=\tau'(i)} =
  \bigcup_{j=0}^{k-1} 
  C_{\tau'^{-1} (j)} \times P_{\tau''(\alpha(j))}
  = 
  \bigcup_{j=0}^{k-1} 
  \pi'^{-1}(\vc\theta'^{-1}(C_j)) \times P_{\tau''(\alpha(j))} \\
   \stackrel{l=\alpha(i)} = 
   \bigcup_{l=0}^{k-1} 
  \pi'^{-1}(\vc\theta'^{-1}(C_{\alpha^{-1}(l)})) \times P_{\tau''(l)}
  =\bigcup_{l=0}^{k-1} 
  \pi'^{-1}(\vc\theta'^{-1}(C_{\alpha^{-1}(l)})) 
  \times 
  \vc\theta''(\pi''(P_{l}))\\
  =\bigcup_{j=0}^{k-1} 
  \pi'^{-1}(\vc\theta'^{-1}(C_{j})) 
  \times 
  \vc\theta''(\pi''(P_{\alpha(j)}))
  =\vc\theta\big(\pi\big( \bigcup_{j=0}^{k-1} 
  C_{j}
  \times 
   P_{\alpha(j)}\big)\big)
   \end{multline*}
for some $(\pi,\vc\theta)$, composed from
$(\pi'^{-1},\pi'^{-1}(\vc\theta'^{-1}))$ 
and $(\pi'',\vc\theta'')$ acting on the corresponding coordinates.
\end{proof}

Hence, for given partitions $\bar C = (C_0,\ldots,C_8)$ 
  and 
  $\bar P = (P_0,\ldots,P_8)$,
  we can restrict our search by considering only permutations
that are representatives of the double-cosets from 
$T(\bar P) \backslash \mathrm{Sym}(9) / T(\bar C) $, where
$T(\bar D) = \{\tau:\ (\pi,\vc\theta,\tau)\in\Aut(\bar D)\mbox{ for some  $\pi$, $\vc\theta$}\}$.

\begin{corollary}\label{c:dcoset}
 For given partitions $\bar C = (C_0,\ldots,C_8)$ 
  and 
  $\bar P = (P_0,\ldots,P_8)$ all permutations~$\tau$
  from the same double-coset in 
  $T(\bar P) \backslash \mathrm{Sym}(9) / T(\bar C) $ 
  result in equivalent concatenated codes.
\end{corollary}

The automorphism groups
of the $65435$ nonequivalent partitions of~$M$ 
and two non\-equiv\-a\-lent partitions
of~$\FF_3^4$
are found in the way described in Section~\ref{ss:iso}.
Using GAP~\cite{GAP},
representatives of all
double-cosets were found
in several hours 
(to fasten the process, we group partitions with 
the same automorphism group and run the double-coset 
calculation once for each such group).
In such a way, we obtain $93278251$
concatenated $1$-perfect $(13,3^{10},3)_3$ codes.
This amount is too huge to check the nonequivalence
using the approach described in Section~\ref{ss:iso}
(
it
takes from less than~$1$ second to several hours for one code, 
depending on its symmetric properties).
However, as we will see below, 
more than~$99.9\%$ of these codes are guaranteedly nonequivalent,
and it remains to process the other~$0.1\%$.

\subsubsection{Uni-concatenated
and multi-concatenated codes}

After permuting the coordinates, a concatenated
$(13,3^{10},3)_3$
code $P$ can lose the property to be concatenated.
However, if the coordinate permutation $\pi$ 
fixes the partition
of the coordinates into two groups, 
$\{0,\ldots,8\}$ and $\{9,10,11,12\}$, 
then the resulting code $\pi(P)$ will 
be surely concatenated.
Indeed, the action of such a permutation 
on the concatenated code can be treated 
as the actions of two coordinate permutations 
on the length-$9$ and length-$4$ 
codes $C_i$ and $P_i$ in the construction~\eqref{eq:P}.
If a coordinate permutation $\pi$ changes the partition
$(\{0,\ldots,8\},\{9,10,11,12\})$,
and $\pi(P)$ is still concatenated,
then the concatenation representation of $\pi(P)$
is not derived from 
 the concatenation representation of $P$;
 we can say in this case that $C$ 
 has more than one concatenation structure,
 or for short, that it is \emph{multi-concatenated}.
 Concatenated codes that are not multi-concatenated
 are called \emph{uni-concatenated}.
 
The equivalence between 
uni-concatenated codes can be recognized 
in an easier way than the equivalence between arbitrary codes.
The following lemma is straightforward.
\begin{lemma}
If two uni-concatenated codes
$$
P= \bigcup_{i=0}^{8} C_i P_{\tau(i)}
\qquad\mbox{and}\quad
D= \bigcup_{i=0}^{8} A_i B_{\gamma(i)}
$$
are equivalent, then
$(C_i)_{i=0}^8$ is equivalent to~$(A_i)_{i=0}^8$
and 
$(P_i)_{i=0}^8$ is equivalent to~$(B_i)_{i=0}^8$.
\end{lemma}
Since, in our classification, we use only one representative
from each equivalence class of RM-like partitions and $1$-perfect partitions,
we can obtain two equivalent uni-concatenated codes only if the ingredient
partitions are the same in the both concatenations.

\begin{lemma}
Two uni-concatenated codes
\begin{equation} \label{eq:PP}
P= \bigcup_{i=0}^{8} C_i P_{\tau(i)}
\qquad\mbox{and}\quad
D= \bigcup_{i=0}^{8} C_i P_{\gamma(i)}
\end{equation}
are equivalent if and only if $\tau$ and~$\gamma$
are in the same double-coset from 
$T(\bar P) \backslash \mathrm{Sym}(9) / T(\bar C)$.
\end{lemma}
\begin{proof}
 Assume that $D$ and~$P$ are equivalent, i.e.,  $D = \alpha(P)$, 
 where $\alpha = (\pi,\vc\theta)$
 for some coordinate permutation
 $\pi=(\pi(0),\pi(1),\ldots,\pi(12))$ 
 and isotopy $\vc\theta=(\theta_0,\ldots,\theta_{12})$.
 By the definition of uni-concatenated codes, 
 $\pi$ fixes the partition $(\{0,\ldots,8\},\{9,10,11,12\})$.
 Hence, $\pi'=(\pi(0),\pi(1),\ldots,\pi(8))$ 
 and $\pi''=(\pi(9)-9,\pi(10)-9,\pi(11)-9,\pi(12)-9)$ 
 are valid permutations of $(0,\ldots,8)$ and $(0,1,2,3)$,
 respectively. 
 Then, denoting $\vc\theta'=(\theta_0,\ldots,\theta_{8})$,
 $\vc\theta''=(\theta_9,\ldots,\theta_{12})$, 
 $\alpha' = (\pi',\vc\theta')$,
 $\alpha'' = (\pi'',\vc\theta'')$,
 we find
 \begin{equation} \label{eq:DD}
  D = \alpha(P)
 = \bigcup_{i=0}^{8} \alpha'(C_i) \alpha''(P_{\tau(i)}).
 \end{equation}
  
(*) \emph{We state that 
there is a permutation~$\beta$ in~$T(\bar C)$
such that $\alpha'(C_i) = C_{\beta(i)}$,
$i=0,\ldots,8$.}
Denote by~$\vc{p}_i$
the word of weight at most~$1$ 
in~$P_i$, $i=0,\ldots,8$.
It follows from~\eqref{eq:PP} that
$$
C_i = 
\{ \vc{c}\in\FF_3^9:\ 
\vc{c}\vc{p}_{\gamma(i)} \in D\},
$$
\begin{equation} \label{eq:fas}
 C_{\gamma^{-1}(i)} = 
\{ \vc{c}\in\FF_3^9:\ 
\vc{c}\vc{p}_{i} \in D\}.
\end{equation}
Denote by $\vc{r}_i$
the word of weight at most~$1$ 
in~$\alpha''(P_i)$, $i=0,\ldots,8$.
It follows from~\eqref{eq:DD} that
$$
\alpha'(C_i) = 
\{ \vc{c}\in\FF_3^9:\ 
\vc{c}\vc{r}_{\tau(i)} \in D\}.
$$
Since 
$\{p_i\}_{i=0}^8 = \{r_i\}_{i=0}^8$,
we have $r_i = p_{\rho(i)}$ 
for some permutation~$\rho$, 
and the last equation
turns to
$$
\alpha'(C_i) = 
\{ \vc{c}\in\FF_3^9:\ 
\vc{c}\vc{p}_{\rho(\tau(i))} \in D\},
$$
$$
\alpha'(C_{\rho^{-1}(\tau^{-1}(i))}) = 
\{ \vc{c}\in\FF_3^9:\ 
\vc{c}\vc{p}_{i} \in D\}.
$$
Comparing with~\eqref{eq:fas},
we find
$$
C_{\gamma^{-1}(i)} =
\alpha'(C_{\rho^{-1}(\tau^{-1}(i))}),
$$
$$
C_{\gamma^{-1}(\tau(\rho(i)))} =
\alpha'(C_{i}),
$$
and so (*) holds with
$\gamma^{-1} \tau \rho $, 
which is in~$T(\bar C)$ 
by the definition of~$T(\bar C)$.

(**) \emph{We state that 
there is a permutation~$\lambda$ in~$T(\bar P)$
such that $\alpha''(P_i) = P_{\lambda(i)}$,
$i=0,\ldots,8$.} The proof is similar to~(*).
Choose a word~$\vc{o}$ in 
$\FF_3^9 \backslash M$, 
where $M=\cup_{i=0}^8 C_i$.
By Proposition~\ref{p:RMl},
for each~$i$ from $\{0,\ldots,8\}$
there is a unique~$\vc{c}_i$ 
in~$C_i$
at distance~$1$ from~$\vc{o}$.
From~\eqref{eq:PP} we find
\begin{equation}\label{eq:fs}
 P_{\gamma(i)} = 
 \{\vc{p}\in\FF_3^4:\ \vc{c}_i\vc{p}
 \in D \}.
\end{equation}
It is easy to see that 
$\cup_{i=0}^8 
C_i= \cup_{i=0}^8 \alpha'(C_i)$,
and so for some permutation~$\rho$ 
we have 
$c_{\rho(i)}\in\alpha'(C_i)$, $i=0,\ldots.,8$.
From~\eqref{eq:DD} we find
$$
 \alpha''(P_{\tau(i)}) = 
 \{\vc{p}\in\FF_3^4:\ \vc{c}_{\rho(i)}\vc{p}
 \in D \},
$$
\begin{equation}\label{eq:fe}
 \alpha''(P_{\tau(\rho^{-1}(i))}) = 
 \{\vc{p}\in\FF_3^4:\ \vc{c}_{i}\vc{p}
 \in D \}.
\end{equation}
From~\eqref{eq:fs} and~\eqref{eq:fe} we conclude that
(**) holds with $\lambda=\gamma\rho\tau^{-1}$.

Now, from~\eqref{eq:PP} and~\eqref{eq:DD},
we have
$$
\bigcup_{i=0}^8 C_iP_{\gamma(i)}
=
\bigcup_{i=0}^8 C_{\beta(i)}P_{\lambda(\tau(i))}
\stackrel{j=\beta(i)}=
\bigcup_{j=0}^8 C_jP_{\lambda(\tau(\beta^{-1}(j)))}
$$
with~$\beta$ from~$T(\bar C)$ and $\lambda$
from~$T(\bar P)$.
We see that $\gamma=\lambda\tau\beta^{-1}$,
which proves the ``only if'' statement.
The ``if'' statement is straightforward.
\end{proof}

So, among the $93278251$ different codes obtained
as shown in the end of Section~\ref{sss:dcoset},
only multi-concatenated 
codes can be equivalent. 
Most of those codes have rank~$12$ and are uni-concatenated
by the following lemma.
\begin{lemma}
A concatenated $(13,3^{10},3)_3$ 
code has rank at most~$12$.
A multi-concatenated $(13,3^{10},3)_3$ 
code has rank at most~$11$.
\end{lemma}
\begin{proof}
Let a $(13,3^{10},3)_3$ code $P$
be represented in the form~\eqref{eq:P}.
Assume without loss of generality that 
$\vc0\in P$.
The union $M=\bigcup_{i=0}^8 C_i$ is a distance-$2$ MDS codes.
Such a code in $\FF_3^9$ is unique up to equivalence, and it is orthogonal
to a word from~$1$s and~$2$s. It follows
that $P$ is orthogonal to a word with nonzeros
in the first~$9$ coordinates and zeros in the last $4$ coordinates.
Hence, the rank of~$P$ is less than~$13$.
If the code is multi-concatenated,
then similarly it is orthogonal to another word with another set of nonzero positions.
Hence, the rank does not exceed~$13-2$.
\end{proof}
For the remaining $1164331$ codes
of rank less than~$12$,
the multi-concatenated
property can be checked relatively fast,
and we found that the majority
of them are uni-concatenated.
Recognizing equivalence
among the  
$74464$ multi-concatenated codes 
(it took about $4.5$ core-years),
$37540$ equivalence classes were found. 
The final results are described in the next section.

\subsubsection{Results}\label{s:res}

\begin{theorem}[computational]\label{th:p13}
 There are exactly $93241327$ equivalence classes
 of concatenated ternary $1$-perfect codes of length~$13$.
\end{theorem}
\begin{table}[h]
\mbox{}\hfill
\begin{tabular}{r|c@{\ \,}c@{\ \,}c@{\ \,}c@{\ }c@{\ }c@{\ }c@{\ }c@{\ }c@{\ }c@{\ }c|c}
dim(kernel) & 0&1&2&3&4&5&6&7&8&9&10&0--10 \\ \hline
rank 10 & -- & -- & -- & -- & -- & -- & -- & -- & -- & -- & 1 & 1 \\
rank 11
&
--
&
--
&
--
&
--
&
693021
&
447241
&
23418
&
634
&
15
&
1 
&
--
&
1164330
\\
rank 12
&
\color{gry}\textbf0
&
\color{gry}\textbf0
&
\color{gry}\textbf0
&
193689
&
70784858
&
20371138
&
719384
&
7919
&
8
&
--
&
--&92076996
\\
rank 13
&
\color{gry}\textbf0
&
\color{gry}\textbf0
&
\color{gry}\textbf0
&
0
&
0
&
0
&
0
&
0&
--
&
--
&
--
&0
\\ \hline
rank 10--13
&
\color{gry}\textbf0
&
\color{gry}\textbf0
&
\color{gry}\textbf0 &
193689&
71477879&
20818379&
742802&
8553&
23&
1&1&93241327
\end{tabular}
\hfill\mbox{}
 \caption{The number of equivalence classes of concatenated
 ternary $1$-perfect codes of length $13$ for each admissible rank and
 kernel dimension.}\label{t:rk}
\end{table}
The distribution of equivalence classes according
to the rank and the dimension of the kernel is shown in Table~\ref{t:rk}.
The mark ``--'' in the table denotes that codes with the corresponding
parameters do not exist even without the restriction to
be concatenated.
In all such cases, there is a theoretical explanation:
\begin{itemize}
 \item a code of rank~$10$ is linear and has kernel dimension~$10$,
 and vice versa;
 \item by Corollary~\ref{c:r-k}, a code of rank~$11$
       has kernel dimension at least~$4$;
 \item the following argument is a special case of~\cite[Proposition~5.1]{PheVil:q-kernel}:
if a ternary code~$C$ has the kernel of size~$|C|/3$ 
(in our case, kernel dimension~$9$), 
then the affine span of~$C$ 
has size $3\cdot |C|$ (i.e., rank~$11$ in our case);
 \item codes of rank~$13$ and kernel dimension~$8$
 do not exist because of the nonexistence of 
 a full-rank tiling of~$\FF_3^5$~\cite{OstSza:2007}
 (for the connection between tilings and $1$-perfect codes, 
 see~\cite{Krotov:3tiling}).
\end{itemize}
Taking into account recently discovered 
length~$13$ perfect ternary codes of full rank
and kernel dimension from~$3$ to~$7$~\cite{Krotov:3tiling},
only the existence of
$1$-perfect $(13,3^{10},3)_3$ codes
of kernel dimension less than~$3$
remains open (in Table~\ref{t:rk},
the corresponding values are grayed).
In particular, 
we see that with concatenation, for $(13,3^{10},3)_3$ codes
of rank~$11$, one can obtain any kernel dimension from~$4$
to~$9$. 
In contrast, 
by the fixed-coordinate switching from the Hamming code,
only 
kernel dimensions~$8$ and~$9$ 
 can be obtained
for these parameters, see~\cite[Table~1]{PheVil:q-kernel}.
Examples of $3$-ary length-$13$ 
perfect codes 
for each known values 
of the rank
and the kernel dimension
(including non-concatenated rank-$13$ codes) are available in~\cite{Perfect-related}.

The distribution of equivalence classes according
to the order of the automorphism group is shown in Table~\ref{t:a}. 
Note that the order of the automorphism group was calculated directly 
(see Section~\ref{ss:iso}) for
multi-concatenated codes,
while for
uni-concatenated codes it was found
from the automorphism group orders of the partitions~$\bar C$ and~$\bar P$ 
and the size of the corresponding
double-coset.

\begin{table}
$$
\begin{tabular}{|@{\ }r@{\ }|@{\ }l@{\ }|}
\hline $|\Aut|$ & \# \\ \hline
27& 49195\\
54& 24928\\
81& 60630474\\
108& 1887\\
162& 3120437\\
216& 46\\
243& 24257914\\
324& 24277\\
486& 1588122\\
648& 308\\
\hline
\end{tabular}\
\begin{tabular}{|@{\ }r@{\ }|@{\ }l@{\ }|}
\hline $|\Aut|$ & \#  \\ \hline
729& 3034912\\
972& 24487\\
1296& 1\\
1458& 222834\\
1944& 439\\
2187& 202868\\
2916& 10047\\
3888& 3\\
4374& 30442\\
5832& 311\\
\hline
\end{tabular}\
\begin{tabular}{|@{\ }r@{\ }|@{\ }l@{\ }|}
\hline $|\Aut|$ & \#  \\ \hline
6561& 8666\\
8748& 2601\\
11664& 3\\
13122& 4521\\
17496& 141\\
19683& 167\\
26244& 634\\
34992& 4\\
39366& 348\\
52488& 52\\
\hline
\end{tabular}\
\begin{tabular}{|@{\ }r@{\ }|@{\ }l@{\ }|}
\hline $|\Aut|$ & \#  \\ \hline
59049& 9\\
69984& 1\\
78732& 135\\
104976& 6\\
118098& 31\\
157464& 13\\
209952& 2\\
236196& 21\\
314928& 1\\
354294& 14\\
\hline
\end{tabular}\
\begin{tabular}{|@{\ }r@{\ }|@{\ }l@{\ }|@{\ }l@{\ }|@{\ }r@{\ }|}
\hline $|\Aut|$ & \# & R & K  \\ \hline
472392& 5  & 11 & 7,8 \\
708588& 10 & 11 & 7,8 \\
1062882& 1 & 12 & 8 \\
1417176& 3 & 11 & 7,8 \\
1889568& 2 & 11 & 8 \\
2834352& 1 & 11 & 8 \\
4251528& 1 & 11 & 8 \\
6377292& 1 & 11 & 9 \\ 
8503056& 1 & 11 & 8  \\
663238368& 1 & 10 & 10\\ \hline
\end{tabular}
$$
 \caption{The number of equivalence classes of concatenated
 ternary $1$-perfect codes of length $13$ for each admissible order 
 of the automorphism group (for some codes, the rank~R and the kernel dimension~K are shown).}\label{t:a}
\end{table}

We finalize this section with 
two particular questions
regarding characteristics of unrestricted $(13,3^{10},3)_3$ codes.
Note that if the answer to the first question
is ``no'', then the second one has answer~``$27$''.

\begin{problem}
 Do there exist ternary $1$-perfect codes of length~$13$, rank~$12$ or~$13$, and
with kernel dimension~$2$, $1$, $0$?
What is the minimum number of automorphisms 
of a ternary $1$-perfect code of length~$13$?
\end{problem}

\section{Conclusion}
In this paper, we studied ternary $1$-perfect codes,
mainly focusing on the classification results.
The two main results of the paper illustrate the two main approaches in 
constructing nonlinear $1$-perfect codes, the switching approach and the concatenation.
(It should be noted that there are also algebraic ways to construct $1$-perfect codes;
for example, one can construct ternary $\mathbb{Z}_3\mathbb{Z}_9$-linear perfect codes as shown in~\cite{Wu:additive}.)
We 
theoretically characterized $1$-perfect codes of rank~$+1$ of any admissible length
and
obtained a computer-aided enumeration 
of the equivalence classes 
of concatenated $1$-perfect 
codes of length~$13$.
The rest of this section contains concluding remarks
that concern related questions for further investigation.

Our characterization of ternary $1$-perfect codes of rank~$+1$ is in the spirit of similar results
for binary $1$-perfect codes of rank~$+2$
in~\cite{AvgHedSol:class}.
Based on the connection
between $4$-ary length-$n$
and binary length-$3n$ perfect codes
(for example, by concatenation~\cite{Zin1976:GCC},
see, e.g.,~\cite[Remark~2]{SKS:drg} for the concrete mapping),
one can hope that the
$4$-ary $1$-perfect codes of rank~$+1$
can also be
characterized;
this remains actual as an objective for future research.
A~variant of that problem is to find a characterization of $4$-ary $1$-perfect codes
of small ($+1$ or $+2$) $2$-rank, where $2$-rank is the dimension
of the affine span over the subfield~$\FF_2$ 
of~$\FF_4$. 
In contrast to the unique linear Hamming code,
there are nonequivalent additive (i.e., linear over~$\FF_2$, or, equivalently, of $2$-rank~$+0$) $4$-ary $1$-perfect codes 
of the same length~\cite{Lindstrom69}, 
which provides additional difficulties to the characterization
of $4$-ary $1$-perfect codes of small $2$-rank.
A~similar question can also be considered for $1$-perfect
codes in Doob spaces~\cite{SHK:additive},
which have much in common with the $4$-ary Hamming space.


The evaluation of the number of equivalence classes of $1$-perfect codes of length~$13$
and rank~$+1$ shows that their number (more than $20$ millions) 
is too large to enumerate them computationally 
using the straightforward approach.
However, our experience with concatenated codes shows 
that combining theoretical
and computational approaches can help to enumerate 
much larger number of equivalence classes of $(13,3^{10},3)_3$ codes.
So, we hope that with developing the theory, together with improving 
the graph isomorphism software and growing the performance of computers,
the enumeration of codes of limited rank or even all $1$-perfect 
$(13,3^{10},3)_3$ codes might be possible.
For studying ternary \mbox{$1$-per}\-fect codes of larger lengths, only theoretical results
can be applied, and among the interesting problems we mention the problem
of characterizing all admissible pairs (rank, kernel dimension) of ternary $1$-perfect codes, which was done for binary codes in~\cite{ASH:RankKernel}.
Another challenge is the problem 
of existence of an $(n=\frac{q^m-1}{q-1}-1,q^{n-m},3)_q$ code that is not 
a shortened $1$-perfect code. Such codes were found for $q=4$~\cite{SWK:2022},
but the ternary case, including the parameters
$(12,3^9,3)_3$, remains unsolved.
Finally, agreeing that the binary case is the most interesting
among $q$-ary $1$-perfect codes,
we believe that the ternary \mbox{$1$-per}\-fect codes also deserve the separate study.

\section*{Data availability}
The dataset containing the results
of the classifications described in Sections~\ref{ss:RM},
\ref{ss:RM9},
\ref{ss:p9}, and~\ref{ss:cc}
is available in the IEEE DataPort repository~\cite{Perfect-related}.

\section*{Acknowledgements}
The authors thank Vladimir Potapov and Merc\`e Villanueva for useful discussions and
the Siberian Supercomputer Center
(SSCC ICMMG SB RAS)
for provided computational resources.
%

\providecommand\href[2]{#2} \providecommand\url[1]{\href{#1}{#1}}
  \def\DOI#1{{\small {DOI}:
  \href{http://dx.doi.org/#1}{#1}}}\def\DOIURL#1#2{{\small{DOI}:
  \href{http://dx.doi.org/#2}{#1}}}

\end{document}